\documentclass[12pt]{amsart}%
\usepackage{amsmath}
\usepackage{amsfonts}
\usepackage{amssymb}
\usepackage{graphicx}
\usepackage{comment}
\usepackage[margin= 1.5 in]{geometry}
\usepackage{setspace}
\usepackage[utf8]{inputenc}
\usepackage{amsmath,amsthm, amssymb}
\usepackage{pdfrender,xcolor}
\usepackage{tcolorbox}
\usepackage{xcolor}%

\providecommand{\U}[1]{\protect\rule{.1in}{.1in}}
\providecommand{\U}[1]{\protect\rule{.1in}{.1in}}
\providecommand{\U}[1]{\protect\rule{.1in}{.1in}}
\providecommand{\U}[1]{\protect\rule{.1in}{.1in}}
\providecommand{\U}[1]{\protect\rule{.1in}{.1in}}
\newtheorem{theorem}{Theorem}

\newtheorem{conjecture}[theorem]{Conjecture}
\newtheorem{corollary}[theorem]{Corollary}

\newtheorem{example}[theorem]{Example}

\newtheorem{lemma}[theorem]{Lemma}

\newtheorem{proposition}[theorem]{Proposition}
\newtheorem{remark}[theorem]{Remark}

\allowdisplaybreaks[1]

\theoremstyle{definition}
\numberwithin{equation}{section}

\newcommand{\R}{\mathbb R}
\newcommand{\Z}{\mathbb Z}

\newcommand{\resumename}{R\'esum\'e}

\begin{document}
\date{\today}
\title{HRT Conjecture and linear independence of translates on the Heisenberg group}
\author{B. Currey, V. Oussa}

\begin{abstract}
	We prove that the HRT (Heil, Ramanathan, and Topiwala) conjecture is equivalent to the conjecture that finite translates of square-integrable functions on the Heisenberg group are linearly independent.
\end{abstract}

\maketitle
\section{Preliminaries and overview of the paper} 

Given $x, y \in \R$, define unitary operators. $T_x$ and $M_y$ by
$$
T_xf(t) = f(t-x), \ \ \ M_y f(t) = e^{2\pi i ty} f(t).
$$
The following conjecture known as the HRT conjecture
\cite{heil2006linear,balan2010almost,bownik2013linear,heil2015hrt,benedetto2015linear,okoudjou2017extension,grochenig2015linear}
is an open problem deeply rooted in time-frequency analysis. It was posed about twenty years ago by Chris Heil, Jay Ramanathan, and Pankaj Topiwala in \cite{heil1996linear} as follows
\begin{conjecture}
\label{HRT}(The HRT Conjecture) Let $\phi\in L^{2}\left(\mathbb{R}\right)  ,\phi\neq0,$ and let $\mathcal F$ be a finite subset of $\R^2$. 
Then the set 
$$
\{ M_yT_x\phi : (x,y) \in \mathcal F\}
$$
 is linearly independent in $L^2(\R)$. 

\end{conjecture}

\

Although the HRT conjecture is still unresolved, there are quite a few results that might be regarded as evidence for an affirmative answer. One substantial contribution in the literature is due to Linnell. In \cite{linnell1999neumann}, Linnell proves that for nonzero $\phi \in L^2(\R),$ $\{ M_yT_x\phi : (x,y) \in \mathcal F\}$ is linearly independent when $\mathcal F$ is a subset of a full-rank lattice of the time-frequency plane. For a full account of partial results available in the literature, we refer the interested reader to \cite{heil2015hrt}.

\

As is well-known, this conjecture can be recast in terms of the Heisenberg group. First, observe that 
\begin{equation}\label{CCR} 
T_x M_y = e^{-2\pi i xy} M_yT_x
\end{equation}
holds for all $x,y\in \R$. Second, 
the joint action of the operators $T_x$ and $M_y$ is irreducible, in the following sense. 

\begin{lemma} \label{irreducible} Let $\mathcal{H}\subset L^{2}\left(\mathbb{R}\right)  $ be a closed and non-trivial subspace which is stable under all the operators $T_x$ and $M_y, x,y \in \R$. Then $\mathcal H = L^2(\R)$.

\end{lemma}

\begin{proof} Fix a nonzero vector $\phi\in\mathcal{H}$ and suppose that $f
\in L^{2}\left(\mathbb{R}\right)$ is orthogonal to the set
$$
\{ M_yT_x \phi: x, y \in \R\} .
$$
We aim to show that $f$ is the zero element in $L^2(\R)$. Now
\[
\int_{\mathbb{R}}e^{-2\pi ity }\phi\left(  t-x\right)
\overline{f\left(  t\right)  }dt=0
\]
for all $x, y \in \R$, so for each $x \in \R$, 
the Fourier transform of the function $t\mapsto\phi\left(  t-x\right)  \overline{f\left(  t\right)
}$ is identically zero, and hence 
\[
0=\int_{\mathbb{R}}\left(  \left\vert \phi\left(  t-x\right)  \right\vert
^{2}\cdot\left\vert f\left(  t\right)  \right\vert ^{2}\right)  dt\text{ for
all }x\in\mathbb{R}\text{.}%
\]
By Fubini's theorem, 
\begin{align*}
0  &  =\int_{\mathbb{R}}\left(  \int_{\mathbb{R}}\left\vert \phi\left(
t-x\right)  \right\vert ^{2}\cdot\left\vert f\left(  t\right)  \right\vert
^{2}dt\right)  dx\\
&  =\int_{\mathbb{R}}\left\vert f\left(  t\right)  \right\vert ^{2}%
\cdot\left(  \int_{\mathbb{R}}\left\vert \phi\left(  t-x\right)  \right\vert
^{2}dx\right)  dt\\
^{x\mapsto t-x}  &  =\left(  \int_{\mathbb{R}}\left\vert f\left(  t\right)
\right\vert ^{2}dt\right)  \cdot\left(  \int_{\mathbb{R}}\left\vert
\phi\left(  x\right)  \right\vert ^{2}dx\right)  .
\end{align*}
Since $\phi$ is nonzero, we have $ \| f\|= 0$, as desired. 
\end{proof}

That the relation \eqref{CCR} is canonical among jointly irreducible two-parameter families of operators is the content of the Stone-von Neumann Theorem, proved independently by Stone and von Neumann in the late 1920's.

\begin{theorem} \label{SVN} {\rm (Stone-von Neumann)} Let $x \mapsto A_x$ and $y \mapsto B_y$ be unitary representations of the additive group $\R$ acting in a Hilbert space $\mathcal H$ such that
for each $x, y \in \R$, 
\begin{equation}\label{CCR2}
A_x B_y  = e^{-2\pi i  xy} B_y A_x.
\end{equation}
Suppose further that
 $\mathcal H$ admits no non-trivial, proper, closed subspace that is invariant under all operators $A_x,B_y, x, y \in \R$. 
Then there is a unitary map $U: \mathcal H \rightarrow L^2(\R)$ such that for all $x, y \in \R$, 
$$
UA_x U^{-1} = T_x, \ \ \ U B_y U^{-1} = M_{y}
$$
\end{theorem} 
The three-dimensional Heisenberg group $\mathbb H$ can be defined as a subgroup of unitary operators on $L^2(\R)$: 
$$
\mathbb{H}=  \{ zM_yT_x : y,x \in \R, z \in \mathbb T\}. 
$$
When $\mathbb{H}$ is identified with $\mathbb{T\times R}\times\mathbb{R}$ in the obvious way, the group operation is given by
\[
\left(  z_{1},y_{1},x_{1}\right)  \left(  z_{2},y_{2},x_{2}\right)  =\left(
z_{1}z_{2}e^{-2\pi i\left(  x_{1}y_{2}\right)  },y_{1}+y_{2},x_{1}%
+x_{2}\right),
\]
where $\left(  z_{1},y_{1},x_{1}\right)  ,\left(  z_{2},y_{2},x_{2}\right)
\in\mathbb{T\times R}\times\mathbb{R}\text{.}$
With the usual topology on $\mathbb{T\times
R}\times\mathbb{R}$, $\mathbb H$ is a
connected topological group with center%
\[
Z=\left\{  \left(  z,0,0\right)  :z\in\mathbb{T}\right\}  .
\]
Moreover, $\mathbb{H}$ is a unimodular group and 
Lebesgue measure on $\mathbb{T}\times\mathbb{R}\times\mathbb{R}$ is a left-invariant measure on the group. We remark that $\mathbb{H}$ is sometimes called the reduced Heisenberg group so as to distinguish it from the simply connected Heisenberg group $\tilde{\mathbb H}=\R^3 $, whose group operation is such that the canonical covering map $(u,y,x) \mapsto (e^{2\pi i u}, y,x)$ is a homomorphism.



Next we recall a few facts about unitary representations of $\mathbb H$. 
A strongly continuous unitary representation $ \pi : \mathbb H \rightarrow \mathcal U(\mathcal H) $, denoted by $(\pi,\mathcal H)$, is said to be irreducible if  $\mathcal H$ admits no non-trivial, proper, closed subspace that is invariant under all operators $\pi(z,y,x)$. As an example, let $k \in \Z\setminus \{0\}$ and for each $(z,y,x) \in \mathbb H$, put
$$
\pi_k(z,y,x) = z^k M_{ky}T_x.
$$
The  relation \eqref{CCR} shows that $(\pi_k, L^2(\R))$ is a homomorphism of $\mathbb H$ into the unitary group $\mathcal U(L^2(\R))$, and it is easy to check that $\pi_k$ is strongly continuous. Lemma \ref{irreducible} shows that $\pi_k$ is irreducible.

Unitary representations $(\pi,\mathcal H)$ and $(\rho,\mathcal K)$ are equivalent if there is a unitary operator $U : \mathcal H \rightarrow \mathcal K$ such that $U\pi(z,y,x) = \rho(z,y,x) U$ holds for all $(z,y,x) \in \mathbb H$. Formally, $\hat{\mathbb H}$ is the space of all equivalence classes of unitary irreducible representations of $\mathbb H$. The following is almost immediate.


\begin{corollary}\label{SVNCor} Let $\mathcal H$ be a Hilbert space and $ (\pi ,\mathcal H) $ an irreducible unitary representation of $\mathbb H$ such that $\pi |_Z$ is non-trivial. Then there is $k \in \Z\setminus \{0\}$ such that $(\pi, \mathcal H)$ is equivalent with $(\pi_k, L^2(\R))$.

\end{corollary}

\begin{proof} As a consequence of Schur's Lemma, the restriction of $\pi$ to $Z$ consists of unitary scalar operators $\pi(z,0,0) = \varphi(z) \mathbf{Id} |_\mathcal H$. Since $z \mapsto \varphi(z)$ is a non-trivial homomorphism of $\mathbb T$, we have $k\in \Z\setminus \{0\}$ such that $\varphi(z) = z^k, z \in \mathbb T$. Now let $A_x = \pi(0,0,x)$ and $B_y = \pi(0,y/k,0)$. The group operation in $\mathbb H$ shows that \eqref{CCR2} holds for each $x,y$, and hence by Theorem \ref{SVN}, there is $U : \mathcal H \rightarrow L^2(\R)$ with $T_xU = UA_x$ and $M_yU = UB_y$. Since $B_y^k = \pi(0,y,0)$ and $M_y^k = M_{ky}$, we get $U \pi(z,y,x) = \pi_k(z,y,x)U$ as desired. 
\end{proof}

Now suppose that $(\pi,\mathcal H)$ is an irreducible unitary representation of $\mathbb H$ that vanishes on $Z$ and let $p:\mathbb H\rightarrow H/Z$ be the canonical quotient map. Then $\pi$ defines a unitary representation $\overline{\pi}$ of $H/Z$ so that $\pi = \overline{\pi} \circ p$. Since $H/Z$ is just the additive group $\mathbb R^2$, then (again by Schur's Lemma \cite[Proposition 3.5]{MR1397028}) we have $\mathcal H =\mathbb C$ and there is $\omega\in \R^2$ such that 
$$
\pi(z,y,x)= \chi_\omega(z,y,x)=e^{2\pi i \omega\cdot (x,y)}.
$$
Define
\[
\Sigma=\left\{  \chi_{0,\omega}:\omega\in%
\mathbb{R}
^{2}\right\}  \overset{\cdot}{\cup}\left\{  \pi_{k}:k\in%
\mathbb{Z}
\backslash\left\{  0\right\}  \right\}
\]





\begin{corollary} Each irreducible representation of $\mathbb H$ is equivalent with exactly one element of $\Sigma$.

\end{corollary}

\begin{proof} We have just shown that each unitary irreducible representation is equivalent with some element of $\Sigma$.  It remains to observe that for $k_1 ,k_2\in \mathbb Z\setminus\{0\}$, 
$k_1 \ne k_2$ implies that $\pi_{k_1}$ and $\pi_{k_2}$ are not equivalent. Similarly, $\omega_1 \ne \omega_2$ implies $\chi_{\omega_1}$ and $ \chi_{\omega_2}$ are inequivalent. 
\end{proof}

It is now clear that Conjecture \ref{HRT} is equivalent with the following. 


\begin{conjecture}
\label{HRT2}(Restatement of HRT) Let $k$ be any nonzero
integer. Let $\phi\in L^{2}\left(
\mathbb{R}
\right)  ,\phi\neq0,$ and  let $\mathcal F$ be a finite subset of $\mathbb H$ such that the cosets $h Z, h \in \mathcal F$ are distinct. 
Then the set 
$$
\{ \pi_k(h)\phi :h  \in \mathcal F\}
$$
 is linearly independent in $L^2(\R)$. 

\end{conjecture}

The purpose of this note is to show that the Conjectures \ref{HRT} and \ref{HRT2} are equivalent with the conjecture that translates in the Heisenberg group are independent. For  $h,k \in\mathbb{H}$ and $F$ in  $C_c(\mathbb H)$, put   
\[
L_k  F\left(  h\right)  =F\left(  k^{-1}h\right)  .
\]
Then for each $k \in \mathbb H$, $L_k$ extends to a unitary operator on $L^2(\mathbb H)$.

\begin{conjecture}
\label{HT}(The Heisenberg-Translate Conjecture) Let $F$ in $L^{2}\left(
\mathbb{H}\right)  ,F\neq0,$ and let $\mathcal F$ be a finite subset of $\mathbb H$, 
such that the cosets $h Z, h \in \mathcal F$ are distinct. 
Then the
collection of vectors $\left\{  L_h  F:  h \in \mathcal F \right\}  $
is linearly independent in $L^{2}\left(  \mathbb{H}\right)  .$

\end{conjecture}

The following remark due to Rosenblatt \cite{rosenblatt2008linear} shows the necessity of the assumption that the cosets $hZ, h \in \mathcal F$ are distinct. 

\begin{remark}
Choose a point $z\in \mathbb T$ 
of $H$ 
such that $z$ has a finite order $n$, and let $K$ be a compact subset
of $\mathbb{H}.$ Put
\[
F =\sum_{\ell=1}^{n}L_{z^{\ell}} 1_{K}.
\]
Then for a fixed natural number $m,$ the following is clearly true
\[
L_{z^m} F=\sum_{\ell=1}^{n}L_{z^{\ell}}  1_{K}=F.
\]

\end{remark}

The primary objective of this note is to prove the following.

\begin{theorem}
\label{main}The HRT conjecture fails if and only if the Heisenberg-Translate
conjecture fails.
\end{theorem}

Let $C_c(\mathbb H / Z) = \{ F \in C_c(\mathbb H ) : L_zF = F, z \in Z\}$; note that $C_c(\mathbb H )$ projects onto $C_c(\mathbb H / Z) $ by 
$$
P : F\mapsto \int_{\mathbb T} F(\cdot \ (z,0,0)) dz. 
$$
It is easily seen that for $p >1$,  $\| PF\|_p\le \| F\|_p$, so $P$ extends to a continuous map with image $L^p(\mathbb H / Z)$, the closure of $ C_c(\mathbb H / Z)$ in $L^p(\mathbb H)$. 
Of course $L^p(\mathbb H / Z)$ is canonically isomorphic with $L^p(\R^2)$. 
It is worth noting that when $p$ is greater than $4$, the analog of Conjecture
\ref{HT} fails.


\begin{proposition}\cite[Theorem
9.18]{heil2006linear} 
\text{ } \label{second}\text{ }

\begin{enumerate}
\item Let $\mathcal F$ be a finite subset of $\mathbb H$, such that the elements $h Z, h \in \mathcal F$ are distinct elements of $
\mathbb H / Z$.  If $F \in L^p(\mathbb H/ Z)$ is non-zero and $1\leq p\leq4$, then the  collection of vectors
$\left\{  L_h  F: h \in \mathcal F\right\}  $ is linearly
independent.

\item If $4<p\leq\infty$ then there exist $F\in L^{p}\left(  \mathbb{H} / Z
\right)  $ and a finite set $\mathcal F$ of $\mathbb H$, such that the cosets $hZ, h \in \mathcal F$ are distinct, and 
$\{ L_hF : h \in \mathcal F\}$ is linearly dependent.

\end{enumerate}
\end{proposition}

\begin{proof}
Given $F\in L^{p}\left(  \mathbb{H} / Z\right)  $ and $(z,y,x), (z_j, y_j, x_j) \in \mathbb H$, 
\begin{align*}
F\left(  \left(  z_j,y_{j},x_{j}\right)  ^{-1}\left(  z,y,x\right)  \right)   &  =F\left(  \overline{z}_jz e^{2\pi i x_jy},y-y_{j},x-x_{j}\right) \\
&  =F\left( 1,  y-y_{j},x-x_{j}\right)  .
\end{align*}
Thus, for complex numbers $c_{1},\cdots,c_{n},$
\[
\sum_{j=1}^{n}c_{j}F\left(  \left(  z_j,y_{j},x_{j}\right)
^{-1}\left(  z,y,x\right)  \right)  =0
\]
if and only if
\[
\sum_{j=1}^{n}c_{j}F\left(  1, y-y_{j},x-x_{j}\right)  =0.
\]
The results of this proposition follow from a straightforward application of
\cite[Theorem 9.18]{heil2006linear} which is due to the work of Rosenblatt and
Edgar \cite{edgar1979difference,rosenblatt1995linear}. In fact, a function
satisfying the claim of the second part of the proposition can be constructed
as follows. Define
\[
F\left(   z,y,x\right)
=\int_{1/3}^{2/3}e^{i\left\langle \left[
\begin{array}
[c]{c}%
x\\
y
\end{array}
\right]  ,\left[
\begin{array}
[c]{c}%
\arccos\left(  t\right) \\
\arccos\left(  1-t\right)
\end{array}
\right]  \right\rangle }dt.
\]
 It is shown in
\cite{edgar1979difference} that
\[
2F\left( 1, y,x\right)  =F\left( 1, y,x+1\right)  +F\left( 1, y,x-1\right)
+F\left( 1, y+1,x\right)  +F\left(  1, y-1,x\right)
\]
and $F$ is a continuous function in $L^{p}\left(  \mathbb{R}^{2}\right) = L^{p}\left(  \mathbb{H}/Z\right)  $.
\end{proof}

\

\section{Proof of Theorem \ref{main}}

We begin with a proof of a standard result; see also \cite{corwin2004representations,folland2016harmonic,moore1973square}.

\begin{lemma}
\label{square_cont}Fix $k\in%
\mathbb{Z}\setminus
\left\{  0\right\} $ and let $f,g\in L^{2}\left(
\mathbb{R}
\right) . $  Then the function $h   \mapsto
\left\langle g,\pi_{k}\left(  h \right)  f\right\rangle $
is continuous and square-integrable on $\mathbb H$.
\end{lemma}

\begin{proof} The fact that $F:  h \mapsto\left\langle
g,\pi_{k}(h)  f\right\rangle $ is continuous
is a consequence of the strong continuity of the representation $\pi_{k}.$ 
The
square-integrability of $F$ is due to the following straightforward
calculations:%
\begin{align*}
\int_{0}^{1}\int_{%
\mathbb{R}
}\int_{%
\mathbb{R}
}\left\vert \left\langle g,\pi_{k}\left(  e^{2\pi i\theta},y,x\right)
f\right\rangle \right\vert ^{2}\text{ }dxdyd\theta &  =\int_{0}^{1}\int_{%
\mathbb{R}
}\int_{%
\mathbb{R}
}\left\vert e^{-2\pi i k\theta}\left\langle g,\pi_{k}\left(  1,y,x\right)
f\right\rangle \right\vert ^{2}\text{ }dxdyd\theta\\
&  =\underset{=1}{\underbrace{\left(  \int_{0}^{1}d\theta\right)  }}\int_{%
\mathbb{R}
}\int_{%
\mathbb{R}
}\left\vert \left\langle g,\pi_{k}\left(  1,y,x\right)
f\right\rangle \right\vert ^{2}\text{ }dxdy\\
&  =\int_{%
\mathbb{R}
}\int_{%
\mathbb{R}
}\left\vert \left\langle g,M_{ky}T_x f\right\rangle
\right\vert ^{2}\text{ }dxdy.
\end{align*}

Now

\begin{align*}
\int_{%
\mathbb{R}
}\int_{%
\mathbb{R}
}\left\vert \left\langle g,M_{ky}T_x
f\right\rangle \right\vert ^{2}\text{ }dxdy&  
 =\int_{%
\mathbb{R}
}\int_{%
\mathbb{R}
}\left\vert \left(\left[M_{-ky} g\right]  \ast f^{\ast
}\right)\left(  x\right)  \right\vert ^{2}\text{ }dxdy.
\end{align*}
In the last equality above, $\ast$ stands for the usual convolution and
$f^{\ast}\left(  x\right)  =\overline{f\left(  -x\right)  }.$ 
For each $y \in \R$, the function $x \mapsto \left(\left[M_{-ky} g\right]  \ast f^{\ast
}\right)\left(  x\right) $ belongs to $C_0(\R)$, and is $L^2$ if and only if 
$$
\widehat {M_{-ky}g} \widehat{f^\ast} : \xi \mapsto \hat g(\xi + ky) \overline{\hat f}(\xi)
$$
belongs to $L^2(\R)$. 
We conclude that
\begin{align*}
\int_{0}^{1}\int_{%
\mathbb{R}
}\int_{%
\mathbb{R}
}\left\vert \left\langle g,\pi_{k}\left(  e^{2\pi i\theta},y,x\right)
f\right\rangle \right\vert ^{2}\text{ }dxdyd\theta &  =\int_{%
\mathbb{R}
}\int_{%
\mathbb{R}
}\left\vert\hat g(\xi + ky) \overline{\hat f}(\xi)
 \right\vert ^{2}\text{ }d\xi dy \\
 &= \int_\R \left( \int_\R \left\vert\hat g(\xi + ky)\right\vert^2dy\right) \vert\hat f(\xi)\vert^2\ d\xi
\\ &  =
k^{-1} \| f\|^2 \| g\|^2 < \infty. 
\end{align*}

\end{proof}

The following result now has a short proof.

\begin{lemma}
\label{forward}If Conjecture \ref{HRT} fails then Conjecture \ref{HT} fails as well.
\end{lemma}

\begin{proof}
Suppose that Conjecture \ref{HRT} fails; then Conjecture \ref{HRT2} fails as well, so we have a nonzero function $\phi \in L^{2}\left(  \mathbb{R}%
\right)  $, elements $h_1, \dots , h_n\in \mathbb H$, and nonzero complex numbers $c_{1},\dots,c_{n}$,  such that the cosets $h_1Z, \dots , h_nZ$ are distinct, and %
\[
\sum_{\ell=1}^{n}c_{j}\pi_{k}(h_{j})  \phi =0.
\]
Put $F(h)  =\left\langle \phi,\pi_{k}\left(  h\right)
\phi \right\rangle .$ Since $\phi$ is a non-zero, according to Lemma
\ref{square_cont}, $F$ is a non-zero element of $ L^{2}\left(
\mathbb{H}\right)  .$ 
Since $\pi_k$ is unitary, we have
\[
0=\left\langle \sum_{\ell=1}^{n}c_{j}\pi\left( h_j\right)
\phi ,\pi_{k}\left(  h\right)  \phi\right\rangle =\sum_{j=1}^{n}c_{j
}\left\langle \phi,\pi_{k}\left(  h_j^{-1}h\right)  \phi \right\rangle  = \sum_j  c_j L_{h_j} F(h) . \]
Thus Conjecture \ref{HT} fails.
\end{proof}

\

It is worth noting that Lemma \ref{forward} was also proved in \cite[Proposition 1.1]{Linnell-2017}.

\

The proof of the converse of Lemma \ref{forward} requires a bit more work. 
Note that by Proposition \ref{second}, for the proof of the converse of 
Lemma \ref{forward}, it is enough to consider functions in the closed subspace 
$$
\mathcal K = \ker P = \{ F \in L^2(\mathbb H):PF = 0\}.
$$

\

Let $F \in C_c(\mathbb H)$; for each $k\in \mathbb Z \setminus \{0\}$, define a sesquilinear form $s_k$ on $L^2(\R) \times L^2(\R)$ by
$$
s_k : (f,g) \mapsto \int_\mathbb H \ F(h) \langle \pi_k(h)  f , g\rangle \ dh
$$
Since $F$ is integrable on $\mathbb H$ then $s_k$ is bounded, and hence defines a bounded linear operator $\pi_k(F)$ on $L^2(\R)$: 
$$
s_k(f,g) = \langle \pi_k(F) f, g\rangle
$$
Straightforward computations show that

\medskip
\noindent
(a) for each $h \in \mathbb H$, $\pi_k(L_hF) = \pi_k(h)\pi_k(F)$, and 

\medskip
\noindent
(b) $\pi_k(F)$ is an integral operator with kernel 
$$
K_{k}^{F}\left(  t,x\right)  =\mathcal{F}_{1}\mathcal{F}_{2}F\left(
k,-kt,t-x\right)  .
$$
where $\mathcal{F}_{1}\mathcal{F}_{2}F$ is the partial Fourier transform of $F(z,y,x)$ with respect to the variables $z\in \mathbb T$ and $y \in \R$. 
Since $K_{k}^{F}\left(  t,x\right) \in L^2(\R^2)$, then $\pi_k(F)$ is a Hilbert-Schmidt operator.

\medskip
\noindent
Observe that $L^2(\mathbb H) = \mathcal K \oplus L^2(\mathbb H / Z)$ and $\mathcal D = (I - P) C_c(\mathbb H)$ is dense in $\mathcal K$, since $I-P$ is a projection. 

\begin{proposition} \label{Plancherel extra-light} The map $F \mapsto (|k|^{1/2} \pi_k(F))_{k \in \mathbb Z \setminus \{0\}}$ extends to a linear isometry 
$$
\mathcal K \longrightarrow \bigoplus_{k \in \mathbb Z \setminus \{0\}} \mathbf{HS}(L^2(\R)). 
$$

\end{proposition}

\begin{proof} Let $F \in \mathcal D$; then $0 = PF = \mathcal F_1F(0,\cdot,\cdot)$ so
$$
\| F\|^2_{L^2(\mathbb H)} = \sum_{k\in \mathbb Z \setminus \{0\}} \int_\R\int_\R |\mathcal F_1F(k,y,x)|^2 dydx.
$$
We claim that for each $k$, $ \| \pi_k(F)\|^2_{\mathbf{HS}}  =|k|^{-1}\int_\R\int_\R |\mathcal F_1F(k,y,x)|^2 dydx.$ Recall that $\pi_k(F)$ is given by
$$
\left(\pi_k(F)\phi\right)(t) = \int_\R \ K_{k}^{F}\left(  t,x\right) \phi (x) dx, \ \ \phi \in L^2(\R)
$$
where $K_{k}^{F}\left(  t,x\right)$ is defined as above, so
$$
 \| \pi_k(F)\|^2_{\mathbf{HS}} = \int_\R\int_\R \ |K_{k}^{F}\left(  t,x\right) |^2 \ dtdx =\int_{\mathbb{R}}\int_{\mathbb{R}}\left\vert \mathcal{F}_{1}\mathcal{F}_{2}F\left(  k,-kt,t-x\right)
\right\vert ^{2}dtdx.
$$
Changing variables gives
$$
 \| \pi_k(F)\|^2_{\mathbf{HS}} =\frac{1}{\left\vert k\right\vert }\int_{\mathbb{R}}\int_{
\mathbb{R}}\left\vert \mathcal{F}_{1}\mathcal{F}_{2}F\left(  k,t,x\right)  \right\vert^2 dtdx = \frac{1}{\left\vert k\right\vert }\int_{\mathbb{R}}\int_{
\mathbb{R}}\left\vert \mathcal{F}_{1}F\left(  k,t,x\right)  \right\vert 
^{2}dtdx
$$
as claimed. Thus for all $F \in \mathcal D$, 
$$
\| F\|^2_{L^2(\mathbb H)} = \sum_{k\in \mathbb Z \setminus \{0\}} \int_\R\int_\R |\mathcal F_1F(k,y,x)|^2 dydx =\sum_{k\in \mathbb Z \setminus \{0\}} |k| \  \| \pi_k(F)\|^2_{\mathbf{HS}}. 
$$
\end{proof}

\begin{lemma}
\label{Converse} If Conjecture \ref{HT} fails then Conjecture \ref{HRT} fails
as well.
\end{lemma}

\begin{proof} Suppose that Conjecture \ref{HT} fails: there exists a non-zero function $F$ in $L^{2}\left(
\mathbb{H}\right)  $, elements $h_1, \dots , h_n \in \mathbb H$, and non-zero complex numbers $c_{1},\cdots,c_{n}$, such that the cosets $h_1Z, \dots , h_nZ$ are distinct, and 
\[
\sum_{j=1}^{n}c_j L_{h_j} F =0.\]
Recall that we may assume that $F \in \mathcal K$. 

By Lemma
\ref{Plancherel extra-light}, we have $k\in\mathbb Z \setminus \{0\}$ such that
$$
\|  \pi_k(F) \|^2_{\mathbf{HS}}\neq0
$$
so choose $\phi \in L^2(\R)$ such that $\psi =  \pi_k(F) \phi \ne 0$. 
But
$$
\sum_{j=1}^n c_j \pi_k(h_j) \psi = \sum_{j=1}^n c_j\pi_k(L_{h_j}F) \phi = \pi_k\left(\sum_{j=1}^n c_j L_{h_j }F\right)\phi  = 0, 
$$
showing that Conjecture \ref{HRT2} fails, and hence Conjecture \ref{HRT}, fails. 
\end{proof}

\begin{remark}
The proof of Theorem \ref{main} is a direct application of Lemma \ref{forward}
and its converse: Lemma \ref{Converse}.
\end{remark}

\section{Additional observations on Conjecture \ref{HT}}

Let $\mathcal{B}(L^2(\mathbb{H}))$ be the algebra of bounded linear operators acting on $L^2(\mathbb{H}).$ Next, let $\mathcal{C}(L)$ be the linear space of all bounded operators on $L^2(\mathbb{H})$ commuting with $L_h,h\in \mathbb{H}.$ It is closed under weak limits and taking adjoints, and as such it is a von Neumann algebra.

Define the right regular representation $R$ of $\mathbb{H}$ as follows. For $h\in\mathbb{H},$ we define, a unitary operator acting by right translation
on $L^{2}\left(  \mathbb{H}\right)  $ as $R_{h}F\left(  x\right)
=F\left(  xh\right).$ According to a well-known result of Takesaki, $\mathcal{C}(L)$ is the von Neummann algebra generated by the right regular representation \cite{takesaki1969generalized}.

\

\begin{proposition} \label{cyclic}
	The right regular representation of $\mathbb{H}$ admits a cyclic vector. In other words, there exists a vector $F\in L^2(\mathbb{H})$ such that the linear span of $R_h F,h\in\mathbb{H}$ is a dense subspace of $L^2(\mathbb{H}).$
\end{proposition}

\

For a proof Proposition \ref{cyclic}, we refer the interested reader to a paper of Losert and Rindler \cite{losert1980cyclic} which gives a construction of a cyclic vector for the regular representation of any first countable locally compact group. A non-constructive proof of Proposition \ref{cyclic} can also be found in \cite{greenleaf1971cyclic}.

\

\begin{proposition}
	If $F$ is a cyclic vector for the right regular representation of the Heisenberg group then Conjecture \ref{HT} holds for $F.$
\end{proposition}
\begin{proof}

	Suppose by ways of contradiction that $\sum_{j=1}^{n}c_{j}L_{h_{j}}F=0$ for some nonzero scalars
	$c_{1},\cdots,c_{n}$ and distinct cosets $h_1Z, \dots , h_nZ.$ Then the linear span of the vectors $R_hF,h\in\mathbb{H}$ is a dense subset of $L^2(\mathbb{H})$ contained the kernel of the bounded operator  $J=\sum_{j=1}^{n}c_{j}L_{h_{j}}.$ The continuity of $J$ implies that $J$ is the zero operator in $\mathcal{B}(L^2(\mathbb{H})).$ This gives a contradiction since it is easy to construct a function $F_1\in L^2(\mathbb{H})$ such that $JF_1\neq 0$  (see Proposition \ref{half_line} and Corollary \ref{compact} for example.)
\end{proof}

\begin{proposition}\label{half_line}
	Conjecture \ref{HT} holds for non-trivial functions which are Schwartz in
	the $\left(y,x\right)$-variable and supported on a half-line in the
	$x$-variable. 
\end{proposition}

\begin{proof}
	Let $F$ be a non-zero function on the Heisenberg group, Schwartz
	in the $\left(y,x\right)  $-variable and supported on a half-line in the
	$x$-variable. Suppose that $\sum_{j=1}^{n}c_{j}L_{h_{j}}F=0$ for some nonzero
	scalars $c_{1},\cdots,c_{n}$ and distinct cosets $h_1Z, \dots , h_nZ.$ Without loss of generality, we may assume that
	$F\in\mathcal{K}$. Since the set of compactly supported and continuous
	functions is dense in $L^{2}\left(\mathbb{R}	\right),$ there exist $\phi\in C_{c}\left(  \mathbb{R}%
	\right)$ and a nonzero integer $k$ such that $\pi_{k}\left(  F\right)  \phi$ is nonzero in
	$L^{2}\left(\R	\right)  $. By assumption,
	\[
	0=\pi_{k}\left(  \sum_{j=1}^{n}c_{j}L_{h_{j}}F\right)  \phi=\sum_{j=1}%
	^{n}c_{j}\pi_{k}\left(  h_{j}\right)  \pi_{k}\left(  F\right)  \phi.
	\]
	On the other hand, it is not hard to verify that $\pi_{k}\left(  F\right)  \phi$ is necessarily supported on a half-line in $L^{2}\left(\R\right).$ However, it is known that the time-frequency shifts of such a function must be linearly independent \cite[Proposition 3]{heil1996linear}. This gives a contradiction. 
\end{proof}

\

A straightforward application of Proposition \ref{half_line} gives the following. 
\begin{corollary}\label{compact}
		Conjecture \ref{HT} holds for all non-trivial functions which are in $C_c^{\infty}(\mathbb{H}).$
\end{corollary}

\

\begin{proposition}\label{commute}
	Let $A$ be an invertible operator in $\mathcal{C}(L).$ Then Conjecture \ref{HT} holds for non-trivial functions of the type
	$AF$ where $F$ is Schwartz in the $\left(  y,x\right)  $-variable and is
	supported on a half-line in the $x$-variable. 
\end{proposition}

\begin{proof}
	Suppose that $\sum_{j=1}%
	^{n}c_{j}L_{h_{j}}AF=0$ for some nonzero scalars
	$c_{1},\cdots,c_{n}$ and distinct cosets $h_1Z, \dots , h_nZ.$ Since $A$ commutes with the operators $L_{h_{j}},$ the
	vector $\sum_{j=1}%
	^{n}c_{j}L_{h_{j}}F$ must be in the kernel of $A.$ The fact that
	$A$ is invertible implies that $\sum_{j=1}%
	^{n}c_{j}L_{h_{j}}F=0.$ However, $F$
	is Schwartz in the $\left(  y,x\right)  $-variable and supported on a
	half-line in the $x$-variable. This contradicts Proposition \ref{half_line}. 
\end{proof}

\

We conclude our work by giving an example describing a large class of functions for which Conjecture \ref{HT} holds. 

\begin{example} Given complex numbers $c_1,\cdots c_n,$ it is easy to verify that
\[
A=e^{\sum_{j=1}%
	^{n}c_{j}R_{h_{j}}}=\sum_{k=0}^{\infty}\frac{\left( \sum_{j=1}%
	^{n}c_{j}R_{h_{j}}\right)  ^{k}}{k!}%
\]
is an invertible operator in $\mathcal{C}(L)$. In light of Proposition \ref{commute}, the following is immediate.  The Heisenberg-Translate
	Conjecture holds for any non-zero function of the type $AF$
	where $F$ is Schwartz in the $\left(  y,x\right)  $-variable and supported on
	a half-line in the $x$-variable.
	\end{example}
	
	\section*{Acknowledgements} 
	The second author thanks Professor Kasso Okoudjou for enlightening discussions on the connection between the HRT and HT conjectures.

\end{document}